\documentclass[11pt]{amsart}

\usepackage{amsfonts}
\usepackage{amsmath}
\usepackage{amsthm}
\usepackage{url}
\usepackage[all]{xy}
\usepackage{graphicx}
\usepackage{latexsym}
\usepackage{amssymb}
\usepackage[cp850]{inputenc}
\usepackage{epsfig}
\usepackage{psfrag}
\usepackage{amsfonts}

\newtheorem{sat}{Theorem}[section]		
\newtheorem{lem}{Lemma}[section]
\newtheorem{kor}{Corollary}[section]		\newtheorem{prop}{Proposition}[section]
				\newtheorem*{defi*}{Definition}
			\newtheorem*{bei*}{Example}
\newtheorem*{sat*}{Theorem}				\newtheorem*{kor*}{Corollary}
\newtheorem*{rmk*}{Remark}					

\let\ssection=\section
\renewcommand{\section}{\setcounter{equation}{0}\ssection} 

\newtheorem*{namedtheorem}{\theoremname}
\newcommand{\theoremname}{testing}

\theoremstyle{remark}
\newtheorem*{bem}{Remark}

\newcommand{\BR}{\mathbb R}			
			\newcommand{\BQ}{\mathbb Q}
			\newcommand{\BZ}{\mathbb Z}

		\newcommand{\CF}{\mathcal F}
\newcommand{\CG}{\mathcal G}

\newcommand{\CO}{\mathcal O}		\newcommand{\CP}{\mathcal P}
		
\newcommand{\CS}{\mathcal S}

\newcommand{\actson}{\curvearrowright}
\newcommand{\D}{\partial}

\newcommand{\bs}{\backslash}

\DeclareMathOperator{\Out}{Out}		
\DeclareMathOperator{\SL}{SL}		

\newcommand{\comment}[1]{}

\DeclareMathOperator{\Aut}{Aut}
\DeclareMathOperator{\Stab}{Stab}

\DeclareMathOperator{\vcd}{{vcd}}

\DeclareMathOperator{\cdm}{\underline{cd}}

\DeclareMathOperator{\gdim}{\underline{gd}}

\DeclareMathOperator{\Inn}{Inn}
\DeclareMathOperator{\St}{St}

\begin{document}

\title[]{Dimension invariants of outer automorphism groups}

\author{Dieter Degrijse}
\address{Centre for Symmetry and Deformation, University of Copenhagen}
\email{D.Degrijse@math.ku.dk}
\author{Juan Souto}
\address{IRMAR, Universit\'e de Rennes 1}
\email{juan.souto@univ-rennes1.fr}

\thanks{The first author was supported by the Danish National Research Foundation through the Centre for Symmetry and Deformation (DNRF92).}

\begin{abstract}
The geometric dimension for proper actions $\underline{\mathrm{gd}}(G)$ of a group $G$ is the minimal dimension of a classifying space for proper actions $\underline{E}G$. We construct for every integer $r\geq 1$, an example of a virtually torsion-free Gromov-hyperbolic group $G$ such that for every group $\Gamma$ which contains $G$ as a finite index normal subgroup, the virtual cohomological dimension $\vcd(\Gamma)$ of $\Gamma $ equals $\gdim(\Gamma)$ but such that the outer automorphism group $\Out(G)$ is virtually torsion-free, admits a cocompact model for $\underline E\Out(G)$ but nonetheless has $\vcd(\Out(G))\le\gdim(\Out(G))-r$.
\end{abstract}
\maketitle

\section{Introduction}
Let $G$ be a discrete virtually torsion-free group. A classifying space for proper actions of $G$, also called a model for $\underline{E}G$, is a proper $G$-CW-complex $X$ such that the fixed point set $X^H$ is contractible for every finite subgroup $H$ of $G$. Such a model $X$ is called cocompact, if the orbit space $G \setminus X$ is compact. The geometric dimension (for proper actions) $\underline{\mathrm{gd}}(G)$ of $G$ is the smallest dimension of a model for $\underline EG$. This invariant is bounded from below by the virtual cohomological dimension $\vcd(G)$ of $G$, which is the cohomological dimension of any finite index torsion-free subgroup of $G$, and in many interesting cases these two quantities are in fact equal (e.g. see \cite{KMPN,DMP,lattices,Luck2,Vogtmann,Armart}). On the other hand, there are by now a number of examples showing that $\gdim(G)$ can be strictly larger than $\vcd(G)$ (see \cite{LearyNucinkis,Conchita,DegrijsePetrosyan,LP}). 

Our main goal in this paper is to provide other examples of groups where the geometric dimension is strictly larger than the virtual cohomological dimension. These examples will arise as groups of (outer) automorphisms of certain Gromov-hyperbolic groups. Before going any further, recall that there is equality between geometric and virtual cohomological dimensions for groups of automorphisms of prominent hyperbolic groups such as free groups (\cite{Luck2,Vogtmann}) and surface group (\cite{Armart}). On the other hand, the construction in \cite{LP} can be adapted to yield center free hyperbolic groups $G$ with finite $\Out(G)$ and for which we have $\gdim(G)>\vcd(G)$. For any such group $G$ we have a fortiori $\gdim(\Aut(G))>\vcd(\Aut(G))$. Our groups are different. We will namely consider automorphisms groups of hyperbolic groups which are {\em dimension rigid} in the following sense.

\begin{defi*} 
{\em We say a virtually torsion-free group $G$ is {\em dimension rigid} if one has $\gdim(\Gamma)=\vcd(\Gamma)$ for every group $\Gamma$ which contains $G$ as a finite index normal subgroup.}
\end{defi*}

\begin{bem}
The notion of dimension rigidity arises naturally when one investigates the behaviour of the geometric dimension under group extensions. Indeed, it follows from \cite[Cor. 2.3]{DegrijsePetrosyan} that if $G$ is dimension rigid and
\[   1 \rightarrow G \rightarrow \Gamma \rightarrow Q \rightarrow 1        \]
is a short exact sequence, then $\gdim(\Gamma) \leq \gdim(G)+ \gdim(Q)$.
\end{bem}

We now state our main result.

\begin{sat}\label{main2}
For every $r\geq 0$, there exists a dimension rigid CAT(0) Gromov-hyperbolic group $G$ such that $\Out(G)$ is virtually torsion-free, admits a cocompact model for $\underline E\Out(G)$ and satisfies
$$\vcd(\Out(G))\le\gdim(\Out(G))-r.$$
Finally, the same result holds if we consider $\Aut(G)$ instead of $\Out(G)$.
\end{sat}

Before commenting on the proof of Theorem \ref{main2}, let us point out that there are a number of groups which are dimension rigid in the sense above. For instance, finitely generated abelian groups are dimension rigid. More generally, all elementary amenable groups of type $\mathrm{FP}_\infty$ are dimension rigid (\cite{KMPN}). Also, the solution of the Nielsen realisation problem implies that surface groups are dimension rigid. The same argument applies to free groups as well. Mostow's rigidity theorem implies the dimension rigidity of cocompact irreducible lattices in semisimple Lie groups non isogenous to $\SL_2\BR$ and without compact factors. As we will see in this paper, certain right angled Coxeter groups are dimension rigid as well (Corollary \ref{coxdimrig}). We will prove that, under certain assumptions, dimension rigidity is preserved when taking free products (Proposition \ref{thm-free-product-rigid}). Naively, one could take this last statement to suggest that groups constructed using dimension rigid groups should be dimension rigid. However, already Theorem \ref{main2} disproves this claim: the group $G$ is dimension rigid but $\Out(G)$ is not. In fact, in the course of the proof of the Theorem \ref{main2} we will also encounter other incarnations of this phenomenon. For instance, dimension rigidity is not preserved under direct products.

\begin{sat}\label{thm-bla}
There are dimension rigid groups $G,G'$ such that $G\times G'$ is not dimension rigid.
\end{sat}

The groups $G_1$ and $G_2$ from Theorem \ref{thm-bla} will arise as certain right angled Coxeter groups. In fact, we will recover the following result from \cite{LP}.

\begin{sat}\label{thm-blabla}
For every integer $r\geq 0$, there is a group $G$ which contains a product of right angled Coxeter groups as a finite index normal subgroup, admits a cocompact model for $\underline{E}G$ and for which we have $\gdim(G)\ge\vcd(G)+r$.
\end{sat}

We should give \cite{LP} the appropriate credit: while working on this paper, we had present all the time the examples in that paper. Theorem \ref{thm-bla} and Theorem \ref{thm-blabla} will be proven in Section \ref{sec: constr}. \\

We will now outline the strategy of the proof of our main theorem. First recall that by work of Bestvina-Feighn \cite[Cor. 1.3]{Bestvina-Feighn} and Paulin \cite{Paulin}, every Gromov-hyperbolic group $G$ with infinite $\Out(G)$ splits over a virtually cyclic group. In our case, the group $G$ arises as a free product
$$G=G_1*\dots*G_r$$
of pairwise distinct, dimension rigid, one-ended Gromov-hyperbolic groups satisfying suitable cohomological properties. Actually, each $G_i$ will be a finite extension of a right angled Coxeter group.

We consider the action of $\Out(G)$ on the (spine of the) outer space of free splittings $\CO$ constructed by Guirardel and Levitt \cite{GL}. The space $\CO$ is contractible. This implies that we can compute via a spectral sequence the cohomology of $\Out(G)$ in terms of the cohomology of the stabilizers $\Stab(\sigma)$ of simplices $\sigma$ in the spine of $\CO$. These stabilizers turn out to be virtually isomorphic to direct products $G_1^{n_1}\times\dots\times G_r^{n_r}$ for suitable choices of the exponents $n_1,\dots,n_r$. The groups $G_i$ are constructed so that their top cohomology with group ring coefficients is torsion of orders prime to each other. This implies that the virtual cohomological dimension of the product is smaller than the sum of the dimensions. On the other hand, the groups $G_i$ are also constructed so that the geometric dimension of the product equals the sum of the dimensions of the factors. 

Armed with these bounds on the virtual cohomological and geometric dimensions of the stabilizers we get, by analysing the aforementioned spectral sequence together with a combinatorial argument, that there is simplex $\sigma$ in the spine of $\mathcal{O}$ such that 
$$\vcd(\Out(G))=\vcd(\Stab(\sigma))\leq \underline{\mathrm{gd}}(\Stab(\sigma))-r \leq  \underline{\mathrm{gd}}(\Out(G))-r.$$
At this point it only remains to prove that $G$ is dimension rigid and that $\mathrm{Out}(G)$ admits a cocompact model for $
\underline{E}\mathrm{Out}(G)$. This is done in the last section of this paper.\\

\noindent \textbf{Acknowledgements.} 
The second author would like to thank Vincent Guirar\-del for many interesting conversations. The first author is grateful for a visit to the IRMAR during which the bulk of the work on this paper was done.

\section{Cohomological tools}\label{sec:tools}

In this section we recall the cohomological tools needed in this paper. 
\medskip

Let $G$ be a discrete group. A classifying space for proper actions of $G$, also called a model for $\underline{E}G$, is proper $G$-CW-complex $X$ such that the fixed point set $X^H$ is contractible for every finite subgroup $H$ of $G$. Such a model $X$ is called cocompact, if the orbit space $G \setminus X$ is compact. The geometric dimension (for proper actions) $\underline{\mathrm{gd}}(G)$ of $G$ is the smallest possible dimension of a model for $\underline EG$. Note that if there exists a cocompact model for $\underline{E}G$, then there also exists a cocompact model for $\underline{E}G$ of dimension $\underline{\mathrm{gd}}(G)$. We refer the reader to the survey paper \cite{Luck2} for more details and terminology about classifying spaces for proper actions. The geometric dimension has a certain algebraic counterpart, namely the Bredon dimension for proper actions $\cdm(G)$. Next, we briefly recall the definition of this quantity, and its relation to cohomology with compact support.

\subsection{Bredon cohomology}
Let $\CF$ be the family of finite subgroups of $G$. The \emph{orbit category} $\mathcal{O}_{\CF}G$ is the category whose objects  are left coset spaces $G/H$ with $H \in \CF$ and where the morphisms are all $G$-equivariant maps between the objects. An \emph{$\mathcal{O}_{\CF}G$-module} is a contravariant functor 
$$M: \mathcal{O}_{\CF}G \rightarrow \mathbb{Z}\mbox{-mod}$$
to the category of $\BZ$-modules. The \emph{category of $\mathcal{O}_{\CF}G$-modules}, denoted by $\mbox{Mod-}\mathcal{O}_{\CF}G$, has as objects all the $\mathcal{O}_{\CF}G$-modules and all the natural transformations between these objects as morphisms. One can show that $\mbox{Mod-}\mathcal{O}_{\CF}G$ is an abelian category that contains enough projective modules to construct projective resolutions. Hence, one can construct functors $\mathrm{Ext}^{n}_{\mathcal{O}_{\CF}G}(-,M)$ that have all the usual properties. The \emph{$n$-th Bredon cohomology of $G$} with coefficients $M \in \mbox{Mod-}\mathcal{O}_{\CF}G$ is by definition
\[ \mathrm{H}^n_{\CF}(G,M)\stackrel{\text{def}}= \mathrm{Ext}^{n}_{\mathcal{O}_{\CF}G}(\underline{\mathbb{Z}},M), \]
where $\underline{\mathbb{Z}}$ is the functor that maps all objects to $\mathbb{Z}$ and all morphisms to the identity map. 

The Bredon cohomology of $G$ can be somewhat more concretely expressed with the help of a model $X$ for $\underline EG$. Indeed, the augmented cellular chain complexes $C_{\ast}(X^H)\rightarrow \mathbb{Z}$ of the fixed point sets $X^H$, for all $H \in \mathcal{F}$, assemble to form a projective (even free) resolution $C_{\ast}(X^{-})\rightarrow \underline{\mathbb{Z}}$. We thus have that
\[ \mathrm{H}^n_{\CF}(G,M)= \mathrm{H}^n(\mathrm{Hom}_{\mathcal{O}_{\mathcal{F}}G}(C_{\ast}(X^{-}),M))\stackrel{\text{def}}= \mathrm{H}^n_G(X,M). \]
For more details, we refer the reader to \cite[Section 9]{Luck} and \cite{fluchthesis}.
 
The \emph{Bredon cohomological dimension of $G$ for proper actions}, denoted by $\cdm(G)$ is defined as
\[ \cdm(G) = \sup\{ n \in \mathbb{N} \ | \ \exists M \in \mbox{Mod-}\mathcal{O}_{\mathcal{F}}G :  \mathrm{H}^n_{\mathcal{F}}(G,M)\neq 0 \}. \]
The invariant $\cdm(G)$ should be viewed as the algebraic counterpart of $\gdim(G)$. Indeed, by \cite{LuckMeintrup} (see also \cite{BradyLearyNucinkis}) we have 
$$ \cdm(G)\leq \gdim(G) \leq \max\{3,\cdm(G)\}. $$
Since the cohomological dimension and the Bredon cohomological dimensions are identical for torsion-free groups, it follows that $\vcd(G)\leq \cdm(G)$ when $G$ is virtually torsion-free.\\

\subsection{Compactly supported cohomology}
Let $X$ be $CW$-complex and let $A$ be a subcomplex of $X$. We denote by $\mathrm{H}^{\ast}_c(X)$ and $\mathrm{H}^{\ast}_c(X,A)$ respectively, the compactly supported cohomology of $X$ and the relative compactly supported cohomology of the pair $(X,A)$. In both cases with $\mathbb{Z}$-coefficients. In the presence of a cocompact model $X$ for $\underline{E}G$ one can compute the Bredon cohomological dimension in terms of cohomology with compact support of certain subspaces of $X$. 

\begin{sat}{\rm \cite[Th. 1.1]{DMP}}\label{th: bredon compact support} 
Let $G$ be a group that admits a cocompact model $X$ for $\underline{E}G$. Then
\[    \cdm(G)= \max\{n \in \mathbb{N} \ | \ \exists K \in \CF \ \mbox{s.t.} \ \mathrm{H}_c^{n}(X^K,X^K_\mathrm{sing})  \neq 0\}   \] 
where $\CF$ is the family of finite subgroups of $G$ and where $X^K_\mathrm{sing}$ is the subcomplex of $X^{K}$ consisting of those cells that are fixed by a finite subgroup of $G$ that strictly contains $K$.
\end{sat}
One should think of Theorem \ref{th: bredon compact support} as being the analogue of the formula 
\begin{equation} \label{eq: vcd formula} \vcd(G)=\max \{  n \in \mathbb{N} \ | \ \mathrm{H}^n_c(X) \neq 0 \} \end{equation}
where $X$ is any contractible proper cocompact $G$-CW-complex (e.g.~see \cite[Cor. 7.6]{brown}), $ \mathrm{H}^n_c(X)\cong \mathrm{H}^n(G,\mathbb{Z}[G])$ and $G$ is assumed to be virtually torsion-free.
\medskip

We finish this section with a lemma that will be used later on to analyse the possible difference in behaviour between the virtual cohomological dimension and the Bredon cohomological dimension, when taking direct products. The lemma is well-known to experts and we will therefore only give a sketch of the proof.

\begin{lem}\label{lem: product vcd} Let $G_1,\ldots,G_r$ be a collection of  virtually torsion-free groups and let $\{n_1,\ldots,n_r\}$ be a collection of non-zero natural numbers. Assume that each $G_i$ admits a cocompact $d_i$-dimensional model for proper actions $X_i$ such that $\mathrm{H}_c^{d_i}(X_i)$ is a non-trivial finite group of exponent $n_i$ and
$\mathrm{H}^{d_i}_{c}(X_i,X_{i,{\mathrm{sing}}})$ contains an element of infinite order.
\begin{itemize}
\item[(a)] If all the numbers in $\{n_1,\ldots,n_r\}$ are pairwise coprime, then one has 
\[    \mathrm{vcd}(G_1\times \ldots \times G_r) \leq \sum_{i=1}^{r}d_i -r +1. \]

\item[(b)] If all the numbers in $\{n_1,\ldots,n_r\}$ have a non-trivial common divisor, then
 \[    \mathrm{vcd}(G_1\times \ldots \times G_r) = \sum_{i=1}^{r}d_i  \]
 and 
\[  \mathrm{H}^{d_1 + \ldots + d_r}(G_1\times 
\ldots \times G_r,\mathbb{Z}[G_1 \times \ldots \times G_r])\] is torsion with exponent $\mathrm{gcd}(n_1,\ldots,n_r)$.
\item[(c)] One has
\[   \underline{\mathrm{cd}}(G_1 \times \ldots \times G_r)= \sum_{i=1}^r d_i \]
\end{itemize}
\end{lem}

Here and in the sequel we set, for typographical reasons, $X_{i,{\mathrm{sing}}}$ instead of $(X_i)_{\mathrm{sing}}$. Denote also by $\BZ_n$ the cyclic group of order $n$.

\begin{proof} 
First note that $X=X_1\times \ldots\times X_r$ is a $ \sum_{i=1}^r d_i $-dimensional cocompact model for proper actions for $G=G_1 \times \ldots \times G_r$, that $$\mathbb{Z}_n \otimes \mathbb{Z}_{m}=\mathrm{Tor}_1^{\mathbb{Z}}(\mathbb{Z}_n,\mathbb{Z}_m)= \mathbb{Z}_{\mathrm{gcd}(n,m)}$$ for all integers $n,m\geq 1$, and that if $G$ and $H$ are two groups of type $FP_{\infty}$, there is a K\"{u}nneth formula (e.g. see \cite[Prop. 0.8 and Ex. V.2.2.]{brown}
\[       0 \rightarrow \bigoplus_{p+q=n} \mathrm{H}^p(G,\mathbb{Z}[G])              \otimes \mathrm{H}^p(\Gamma,\mathbb{Z}[H]) \rightarrow \mathrm{H}^{n}(G \times H,\mathbb{Z}[G\times H]) \] \[ \rightarrow \bigoplus_{p+q=n+1}\mathrm{Tor}_1^{\mathbb{Z}}(\mathrm{H}^p(G,\mathbb{Z}[G]),\mathrm{H}^q(H,\mathbb{Z}[H])) \rightarrow 0.                     \]
Using these observations together with (\ref{eq: vcd formula}), both (a) and (b) follow by induction on $r$.

Finally to prove (c), note that we already know that
$$\underline{\mathrm{cd}}(G)\le\dim X= \sum_{i=1}^r d_i.$$
To obtain the other inequality, recall that by \cite[Th. 2.4]{DMP} one has 
$$\mathrm{H}_c^{n}(X_i,X_{i,\mathrm{sing}})=\mathrm{H}^{n}_{\mathcal{F}}(G_i,F_i)$$ 
for the free module $F_i=\mathbb{Z}[-,G_i/e]$. Moreover, under the given finiteness assumptions, there is also a K\"{u}nneth formula in Bredon cohomology (e.g. following \cite[Ch 3, Section 13]{fluchthesis}) that implies 
\[ \mathrm{H}^{d_1+\ldots + d_r}_G(X,F_1\otimes\ldots \otimes F_r) \cong  \mathrm{H}^{d_1}_{G_1}(X_1,F_1)\otimes  \ldots  \otimes \mathrm{H}^{d_r}_{G_r}(X_r,F_r).     \]
The assumption that the $\mathrm{H}_c^{d_i}(X_i,X_{i,\mathrm{sing}})$ contain an infinite order element implies that $\mathrm{H}^{d_1+\ldots + d_r}_{\CF}(G,F_1\otimes\ldots \otimes F_r)= \mathrm{H}^{d_1+\ldots + d_r}_G(X,F_1\otimes\ldots \otimes F_r)$ contains such an element and is a fortiori non-trivial.  Altogether we get that 
$$\underline{\mathrm{cd}}(G)\ge  \sum_{i=1}^r d_i, $$
as we needed to prove.
\end{proof}

\section{Construction Of Certain Dimension rigid groups} \label{sec: constr}

The purpose of this section is to prove the following proposition.

\begin{prop}\label{groups-Gp}
For every odd prime $p$ there exists a dimension rigid one-ended CAT(0) Gromov-hyperbolic group $G_p$ which admits a $3$-dimensional cocompact model for proper actions $X_p$ such that 
$$\mathrm{H}^3_c(X_p)\cong \BZ_p\text{ and }\mathrm{H}_c^3(X_p,X_{p,\mathrm{sing}})\otimes_\BZ\BQ\neq 0.$$
Moreover, $\Out(G_p)$ is finite.
\end{prop}

The group $G_p$ of Proposition \ref{groups-Gp} will be obtained as a semi-direct product
$$W_p\rtimes\BZ_p$$
of a dimension rigid right angled Coxeter group $W_p$ and the cyclic group $\BZ_p$ of order $p$. Before launching the proof of Proposition \ref{groups-Gp}, we use it to prove Theorem \ref{thm-bla} and Theorem \ref{thm-blabla} from the introduction. 

\begin{proof}[Proof of Theorem \ref{thm-bla}]
Let $G_3$ and $G_5$ be provided by Proposition \ref{groups-Gp}. Both groups are dimension rigid by assumption and it follows from Lemma \ref{lem: product vcd} (a) that $\vcd(G_3\times G_5)\le \dim X_3+\dim X_5-1$. Part (c) of the same lemma shows that $\cdm(G\times G')=\dim(X_3)+\dim(X_5)$. Moreover, since $W_3\times W_5$ is a finite index normal subgroup of $G_3 \times G_5$, we conclude that $W_3$ and  $W_5$ are dimension rigid, but $W_3\times W_5$ is not.
\end{proof}

\begin{proof}[Proof of Theorem \ref{thm-blabla}]
The same argument, applied to $G=G_{p_1}\times\dots\times G_{p_{r+1}}$ where $p_1,\dots,p_{r+1}$ are pairwise distinct odd primes, yields that
$$\vcd(G)\le\left(\sum_{i}\vcd(G_{p_i})\right)-r=\cdm(G)-r$$
Moreover, as we mentioned after the statement of Proposition \ref{groups-Gp}, each $G_{p_i}$ contains a certain right angled Coxeter group $W_{p_i}$ as a normal subgroup. It follows that the Coxeter group $W_{p_1}\times\dots\times W_{p_{r+1}}$ is a finite index normal subgroup in $G$.
\end{proof}

Now continue with the same notation and note that from part (b) in Lemma \ref{lem: product vcd} we get that $\vcd(G_p^n)=3n$ and that $H^{3n}(G_p^n,\BZ[G_p^n])\simeq\BZ_p$ where $G_p^n$ is the $n$-fold product of $G_p$ with itself. In particular, the same arguments we just used to prove Theorem \ref{thm-bla} and Theorem \ref{thm-blabla} can also be used to prove the following fact that we state here simply for further reference.

\begin{lem}\label{wiesollichesnennen}
Let $\{p_1,\ldots,p_r\}$ be a finite collection of pairwise distinct odd prime numbers and let $\{n_1,\ldots,n_r\}$ be a collection of natural numbers. Then
$$\mathrm{vcd}( \prod_{i=1}^r G^{n_i}_{p_i}) \leq 3\sum_{i=1}^{r}n_i -r +1,
\text{ but } 
\underline{\mathrm{cd}}(\prod_{i=1}^r G^{n_i}_{p_i})= 3\sum_{i=1}^{r}n_i.$$
Here $  G^{n_i}_{p_i}$ is the $n_i$-fold product with itself of the group $G_{p_i}$ provided by Proposition \ref{groups-Gp}. \qed
\end{lem}

The remainder of this section is devoted to proving Proposition \ref{groups-Gp}. We start by reminding the reader some basic facts about right angled Coxeter groups. For more details and proofs we refer the reader to \cite{DavisBook}.

\subsection{Right angled Coxeter groups}
Let $\CG$ be a finite graph, $L$ the associated flag complex whose $1$-skeleton is $\mathcal{G}$, and set
$$d= \dim(L).$$ 
Denote the vertex and edge sets of $\CG$ by $V=V(\CG)$ and $E=E(\CG)$ respectively. The right angled Coxeter group determined by $\CG$ is the Coxeter group $W=W_\CG$ with presentation
\[       W = \langle V \ | \ s^2 \ \mbox{for all $s \in V$, and
  \ }   (st)^2 \ \mbox{if $(s,t) \in E$} \rangle .  \] 
Note that $W$ fits into the short exact sequence
\begin{equation} \label{eq: torsionfree kernel}  1 \rightarrow N \rightarrow W \xrightarrow{\pi} \bigoplus_{s \in V}\mathbb{Z}_2 \rightarrow 1
\end{equation} 
where $\pi$ takes $s \in V$ to the generator of the $\mathbb{Z}_2$-factor corresponding to $s$. A subset $J \subseteq V$ is called spherical if the subgroup $W_J=\langle J \rangle$ is finite. Equivalently, $J$ is spherical if it spans a simplex in $L$. The empty subset of $J$ is by definition spherical. If $J$ is spherical, then $W_J$ is called a spherical subgroup of $W$ and a coset $wW_J\in W/W_J$ with $J$ is spherical is called a spherical coset. 

Let $\CS$ be the poset of spherical subsets of $V$ ordered by inclusion, identify each $\sigma\in \mathcal{S}$ with the associated simplex of $L$ and let $L\setminus\sigma$ be its complement. Denote by $K$ the geometric realization of $\CS$ and note that $K$ equals the cone over the barycentric subdivision of $L$. In particular, $K$ is contractible. Let $W\CS$ be the poset of all spherical cosets ordered by inclusion and $\Sigma$ its geometric realization. The group $W$ acts by left multiplication on $W\CS$ and thus on $\Sigma$ such that $K$ is a strict fundamental domain for the action $W\actson\Sigma$. The simplicial complex $\Sigma$ is called the {\em Davis complex} of the Coxeter group $W$. It is a proper cocompact $W$-CW-complex. Moreover, since $\Sigma$ admits a complete $W$-invariant CAT(0)-metric, it follows that $\Sigma$ is a cocompact model for $\underline{E}W$ (see \cite[Th. 12.1.1 \& Th. 12.3.4]{DavisBook}). In particular, $\Sigma^{F}$ is non-empty for every finite subgroup $F$ of $W$, meaning that every finite subgroup of $W$ is sub-conjugate to a spherical subgroup of $W$. This implies that the kernel $N$ above is torsion-free, and hence that $W$ is virtually torsion-free.

It is known that the Coxeter group $W$ is Gromov-hyperbolic if and only if any cycle of length 4 in  $\CG$ has at least one diagonal (e.g.~see  \cite[Cor. 12.6.3]{DavisBook}).
The center of a right angled Coxeter group is trivial as long as $\CG$ has radius at least $2$, i.e.~there is no vertex which is connected to every other vertex by an edge. Finally, $W$ is one-ended if and only if $\mathcal{G}$ does not have a separating clique, i.e.~if the complement $L\setminus \sigma$ of every simplex $\sigma$ in $L$ is connected (see \cite[Th. 8.7.2]{DavisBook}).

Since the Coxeter group $W$ is virtually torsion-free, it has a well-defined virtual cohomological dimension $\vcd(W)$. In  \cite[Cor. 8.5.3]{DavisBook} Davis gives a formula to compute $\vcd(W)$ and in \cite[Th. 5.4.]{DMP} it is shown that the virtual cohomological dimension $\mathrm{vcd}(W)$ of $W$ always coincides with the Bredon cohomological $\underline{\mathrm{cd}}(W)$. Combing these references, we immediately obtain the following lemma, where $\overline{\mathrm{H}}^{\ast}$ denotes reduced cohomology with $\mathbb{Z}$-coefficients.

\begin{lem} \label{I should go home}
 If $\overline{\mathrm{H}}^{d}(L)\neq 0$ then 
$\mathrm{vcd}(W)=\underline{\mathrm{cd}}(W)=\dim(\Sigma)=d+1$. If moreover, $\overline{\mathrm{H}}^{d}(L\setminus\sigma)=0$ for every non-empty simplex $\sigma \in \mathcal{S}$, then $\mathrm{H}^{d+1}(W,\mathbb{Z}[W])\cong  \overline{\mathrm{H}}^{d}(L).$ \qed
\end{lem}

Continuing with the same notation, let $\Aut(\CG)$ be the finite group of graph automorphism of $\CG$ and note that $\Aut(\CG)=\Aut(L)$ because $L$ is flag. Since every graph automorphism of $\CG$ naturally gives rise to a group automorphism of $W$, one can form the semi-direct product 
$$\Gamma= W \rtimes\Aut(\CG).$$ 
The action of the Coxeter group $W$ on the poset $W\CS$ extends to an action $\Gamma\actson W\CS$ via
\[     (w,\varphi)\cdot w_0W_J = w\varphi(w_0)W_{\varphi(J)}     \]
and $\varphi(\emptyset)=\emptyset$ for every $(w,\varphi) \in \Gamma$ and every $wW_J \in W\mathcal{S}$. This implies that $\Gamma$ acts properly isometrically and cocompactly on the Davis complex $\Sigma$ of $W$, extending the action of $W$. Note that, since $\Sigma$ is CAT(0), it follows again that $\Sigma$ is a cocompact model for $\underline{E}\Gamma$. We record this fact for later reference.
\begin{lem}\label{I am tired}
The action $W\actson\Sigma$ extends to a proper, cocompact, isometric action $\Gamma= W \rtimes\Aut(\CG)\actson\Sigma$. The Davis complex $\Sigma$ thus becomes a cocompact model for $\underline{E}\Gamma$.\qed
\end{lem}

\subsection{Automorphisms of right angled Coxeter groups}
Continuing with the same notation, we recall a few facts concerning the group of automorphisms $\mathrm{Aut}(W)$ of a right angled Coxeter group $W$. We begin with the following result by Tits.

\begin{sat}[Tits \cite{tits}]\label{th: tits}  
There is a split short exact sequence
\[ 1 \rightarrow \mathrm{Aut}^{0}(W) \rightarrow \mathrm{Aut}(W) \xrightarrow{\pi} \mathrm{Aut}(\mathcal{S}) \rightarrow 0, \]
where $\mathrm{Aut}^{0}(W)$ consists of those automorphisms of $W$ that map each generator $s\in V$ to a conjugate element, and where $\mathrm{Aut}(\mathcal{S})$ is the group of automorphisms of $\mathcal{S}$ viewed as a commutative groupoid under the operation of symmetric difference $u\Delta v=(u \setminus v)\cup (v \setminus u)$, defined whenever all three sets $u,v$ and $u\Delta v$ belong to $\CS$.
\end{sat}

\noindent Given an automorphism $\varphi \in \mathrm{Aut}(W)$, $\pi(\varphi) \in \mathrm{Aut}(\mathcal{S})$ is defined as follows:  for a $J \in \mathcal{S}$, we note that $\varphi(W_J)$ is finite and hence there exist a unique minimal $T=\pi(\varphi)(J) \in \mathcal{S}$ such that $\varphi(W_J)$ is sub-conjugate to $W_T$. Conversely, if $\varphi \in \mathrm{Aut}(\mathcal{S})$, then we one can define an automorphism $i(\varphi)$ of $W$ by setting $i(\varphi)(s)=\prod_{t\in \varphi(s)}t$ for every $s \in V$.  This yields a homomorphism $i: \mathrm{Aut}(\mathcal{S})\rightarrow \mathrm{Aut}(W)$ such that $\pi\circ i=\mathrm{Id}$. 
\medskip

We will return shortly to the group $\Aut(\CS)$, but first note that the group $\Inn(W)$ of inner automorphisms $W$ is contained in $\Aut^0(W)$. M\"{u}hlherr proved that $\Aut^0(W)$ is generated by partial conjugations. Before making this precise, recall that $V$ is the set of vertices of $\CG$ and that the star $\mathrm{St}_{\CG}(s)$ of $s\in V$ is the set of edges adjacent to $s$.

\begin{sat}[M\"{u}hlherr \cite{Mul}]\label{th: mul} 
The group $\mathrm{Aut}^{0}(W)$ is generated by automorphism of the form $\varphi_{s,U}$, where $s\in V$, $U$ is a connected component of $\CG\setminus \mathrm{St}_{\CG}(s)$ and $\varphi_{s,U}(t)=sts$ if $t \in U$ and $\varphi_{s,U}(t)=t$ if $t \notin U$.
\end{sat}

Note in particular that M\"uhlherr's theorem implies that $\Aut^0(W)=\Inn(W)$ as long as $\CG\setminus\St_\CG(s)$ is connected for every vertex $s\in V$. Since $L$ is the flag complex associated to $\CG$, this is the case if and only if $L\setminus\St_{L}(s)$ is connected  for every $s$. We conclude the following.

\begin{lem}\label{lem-inn=aut0}
$\Inn(W)=\Aut^{0}(W)$ whenever $L\setminus\St_{L}(s)$ is connected for every $s\in V$.\qed
\end{lem}

Returning to the other factor, note that the group $\Aut(\CG)$ of graph automorphism of $\CG$ is contained in $\Aut(\CS)$. We prove next that both groups actually agree as long as the graph $\CG$ satisfies some not very stringent conditions.

\begin{lem} \label{lem: groupoid} If $L$ is a finite flag complex such that every non-maximal simplex of $L$ is contained in at least $2$ maximal simplices, then $\mathrm{Aut}(\mathcal{S})=\mathrm{Aut}(\mathcal{G})$.
\end{lem}

Before we start the with the proof of the lemma, note that the condition of the lemma is equivalent that asking that $L$ is a finite flag complex such that every simplex of $L$ is uniquely determined by the set of those maximal simplices it is contained in. Indeed, assume that every non-maximal simplex is contained in at least $2$ maximal simplices. Now let $v$ be a non-maximal simplex, let $M=\{m_1,\ldots,m_r\}$ be the set of distinct maximal simplices that contain $v$ and assume that $v$ is not uniquely determined by $M$, i.e.~$v \subsetneq \bigcap_{i=1}^r m_i$. Denote $w=\bigcap_{i=1}^r m_i$ and let $w'$ be the simplex of $m_1$ determined by all vertices of $v$ and all vertices of $m_1$ that are not vertices of $w$. Then $w'$ is a proper face of $m_1$ that contains $v$ and such that that vertices of $w'$ and $w$ span $m_1$. By our assumptions, $w'$ must be contained in another maximal simplex besides $m_1$. Since $w'$ contains $v$ this other maximal simplex must be some $m_j$, with $j\neq 1$. But since $w$ is contained in $m_j$, this means that $m_1$ is contained in $m_j$. Hence $m_1=m_j$ by maximality, which is a contradiction.  On the other hand, if there is a non-maximal simplex that is contained in exactly one maximal simplex, then that non-maximal simplex is certainly not uniquely determined by the set of those maximal simplices it is contained in.

\begin{proof}

As $\Aut(L)=\Aut(\CG)\subset\Aut(\CS)$, it suffices to prove that any arbitrary $\phi\in\Aut(\CS)$ induces an automorphism of the complex $L$. Since $\mathcal{S}$ is the set of all simplices of $L$, it suffices to show that $\phi$ preserves inclusions.

Recall that the composition in the groupoid $\CS$ is given by the symmetric difference $u\Delta v=(u \setminus v)\cup (v \setminus u).$ More precisely, if $u,v,u\Delta v\in\CS$ then the composition $u*v$ in $\CS$ of $u$ and $v$ is defined as $u*v=u\Delta v$. The basic observation to keep in mind is the following easily checked fact.
\begin{quote}
Given $u,v \in \mathcal{S}$ we have that $u \Delta v \in \mathcal{S}$ if and only if $u \cup v \in \mathcal{S}$.
\end{quote}
Now suppose that $u,v\in\CS$ are such that $u\subset v$. Then $u*w$ is defined for every  $w\in\CS$ for which $v*w$ is defined. Moreover, if $u\neq v$ then, by the condition on $L$, there is some maximal simplex which contains $u$ but does not contain $v$. It follows that $u*w$ is defined for an strictly larger set of $w$'s than $v*w$. On the other hand, if $u$ is not contained in $v$ there is by the condition in the lemma some maximal simplex $w$ with $v\Delta w\in\CS$ but $u\Delta w\notin\CS$. In other words we have
$$u\subsetneq v\text{ iff }\{w\in\CS\vert u\Delta w\in\CS\}\supsetneq\{w\in\CS\vert v\Delta w\in\CS\}.$$
The automorphism $\varphi\in\Aut(S)$ preserves the right side of this equivalence and thus also the left one. In other words, $\varphi$ preserves inclusions between simplices, as we needed to prove.
\end{proof}

We now come to the main observation we will need below.

\begin{prop}\label{need coffee} 
Let $L$ be a connected finite flag complex such that
\begin{itemize}
\item every non-maximal simplex of $L$ is contained in at least two maximal simplices,
\item the graph $\mathcal{G}=L^1$ has radius at least $2$, and 
\item $\CG\setminus\St_\CG(s)$ is connected for every vertex of $\CG$. 
\end{itemize}
If we denote the right angled Coxeter group determined by $\mathcal{G}$ by $W$, then
\[   \mathrm{Aut}(W)= W \rtimes \mathrm{Aut}(\mathcal{G})  \]
 and every extension of a finite group by $W$ acts properly, isometrically and cocompactly on the Davis complex $\Sigma$ of $W$.
\end{prop}

\begin{proof} 
Recall first that the assumption that $\CG$ has at least radius $2$ implies that the center of $W$ is trivial and hence that we can identify the latter group with the subgroup $\Inn(W)$ of $\Aut(W)$ of inner automorphisms. That fact that $W$ has trivial kernel implies that every extension of a finite group by $W$ maps into $\Aut(W)$ with finite kernel. Hence it suffices to prove that $\Aut(W)$ itself acts properly, isometrically and cocompactly on the Davis complex $\Sigma$. To see that this is the case, note that under the given assumptions, Lemma \ref{lem-inn=aut0}, Lemma \ref{lem: groupoid} and Theorem \ref{th: tits} imply that
$\Aut(W)=\Inn(W)\rtimes\Aut(\CS)=W\rtimes\Aut(\mathcal{G})$.
The proposition now follows directly from Lemma \ref{I am tired}.
\end{proof}

Combining Proposition \ref{need coffee} and Lemma \ref{I should go home}, we obtain the following.

\begin{kor}\label{coxdimrig}
Suppose that $L$ is a $d$-dimensional flag complex that satisfies the conditions from Proposition \ref{need coffee} and such that $\overline{\mathrm{H}}^{d}(L)\neq 0$. Then the right angled Coxeter group $W$ determined by the 1-skeleton $L^1$ of $L$ is dimension rigid.\qed
\end{kor}

\subsection{Proof of Proposition \ref{groups-Gp}}\label{subsec-examples}

We are now ready to prove Proposition \ref{groups-Gp}. To do so, we need to construct for each odd prime number $p$ a group $G_p$. So fix an odd prime $p$. We start considering the action of  the cyclic group $\BZ_p$ by rotations on a regular $p$-gon $A$ and let $P$ be the cone of $A$. In more concrete terms, $P$ is nothing more than a regular $p$-gon considered as a 2-dimensional object. The action $\BZ_p\actson A$ extends to a simplicial action $\BZ_p\actson P$. Now let $Z$ be the quotient of $P$ obtained by identifying points of $A\subset P$ that lie in the same $\BZ_p$-orbit. Note that the space $Z$ is still a $\BZ_p$-CW complex and that the singular set
\[  Z_{{\mathrm{sing}}}= \{ x \in Z \ | \ \mathrm{Stab}_{C_{p}}(x) \neq \{e\} \}   \]
is the disjoint union of the circle $A/C_p$ and a point coming form the center of $P$. Moreover, 
\[        \mathrm{H}^1(Z)=0\text{ and }\mathrm{H}^2(Z)=\BZ_p.    \] 
Now $\BZ_{p}$-equivariantly subdivide the cellular structure on $Z$ to obtain an flag triangulation $L$ of $Z$ whose $1$-skeleton $\CG_p$ has at least radius $2$, such that every cycle of length 4 in $\CG_p$ has at least one diagonal, such that the complement $\CG_p\setminus\St_{\CG_p}(s)$ of the star of every vertex $s$ of $\CG_p$ is connected and such that $\mathcal{G}$ does not have any separating cliques. For instance, one can take $L$ to be the third barycentric subdivision of $Z$. All this implies that the right angled Coxeter group $W_p$ determined by $\CG_p$ is Gromov-hyperbolic, center-free and one-ended. Since we clearly also have that every non-maximal simplex of $L$ is contained in at least $2$ maximal simplices, it follow from Proposition \ref{need coffee} that $W_p$ is dimension rigid with finite $\Out(W_p)$ such that every finite extension of $W_p$ acts isometrically, properly and cocompactly on the Davis complex $\Sigma$ of $W_p$. This applies in particular to the semi-direct product
$$G_p=W_p\rtimes\BZ_p$$
arising from the action $\BZ_p\actson L$. Moreover, The CAT(0)-property of $\Sigma$ ensures that $\Sigma$ is a cocompact model for $\underline EG_p$.

Since $L\setminus \sigma$ is homotopic to a one-dimensional space, for any non-empty simplex $\sigma$ of $L$, it follows from Lemma \ref{I should go home} that $\vcd(G_p)=\underline{\mathrm{gd}}(G_p)=3$ and 
\[    \mathrm{H}^3_c(\Sigma)=\mathrm{H}^3(G_p,\mathbb{Z}[G_p])=\mathrm{H}^2(L)=\mathbb{Z}_p.    \]

Since the group $G_p$ contains $W_p$ as a subgroup of finite index, one can immediately conclude that $G_p$ has trivial center, is Gromov-hyperbolic and one-ended. Moreover, by \cite[Lemma 5.4]{GL}, we have that $\mathrm{Out}(G_p)$ is also finite. Since $W_p$ is generated by elements of order $2$ and $p$ is an odd prime, it follows that $W_p$ is not only normal but also characteristic in $G_p$. This implies that if a group $\Gamma$ contains $G_p$ as a normal subgroup of finite index, $W_p$ is also normal in $\Gamma$. We can therefore conclude that $G_p$ is also dimension rigid.

It remains to show that $\mathrm{H}^3_c(\Sigma,\Sigma_{\mathrm{sing}})\otimes_{\mathbb{Z}} \mathbb{Q} \neq 0$, where $\Sigma$ is viewed as a $G_p$-space. Let $K$ be the cone over the barycentric subdivision $L'$ of $L$ and recall that $K$ is a strict compact fundamental domain for the $W_p$-action on $\Sigma$. Moreover, $L'$ is a strict fundamental domain for the singular set of $\Sigma$ viewed as a $W_p$-space. Using these observation and denoting $K_{{\mathrm{sing}}}=\Sigma_{\mathrm{sing}}\cap K $ we deduce that
\[    \Sigma_{\mathrm{sing}}= \bigcup_{w\in W} wK_{\mathrm{sing}}    \]
and

\begin{equation} \label{eq: Ksing}K_{{\mathrm{sing}}}= L' \cup_{L'_{{\mathrm{sing}}}} C(L'_{{\mathrm{sing}}}) \end{equation}
where
\[   L_{{\mathrm{sing}}}=   \{ x \in L \ | \ \mathrm{Stab}_{C_{p}}(x) \neq \{e\} \}       \]
is homeomorphic to $Z_{{\mathrm{sing}}}$ and hence to the disjoint union of a circle and point, and $C(-)$ denotes taking the cone with same conepoint as $K=C(L')$. Now note that the map 
\[     \mathrm{H_c^{\ast}}(\Sigma,\Sigma_{\mathrm{sing}}\cup \bigcup_{w \in W\setminus \{e\}} wK ) \rightarrow \mathrm{H}^{\ast}_c(K,K_{\mathrm{sing}}),   \]
induced by the inclusion of pairs
\[   (K,K_{\mathrm{sing}}) \rightarrow (\Sigma,\Sigma_{\mathrm{sing}}) \rightarrow (\Sigma,\Sigma_{\mathrm{sing}}\cup \bigcup_{w \in W\setminus \{e\}} wK ) ,     \]
precomposed with the isomorphism
\[        \mathrm{H}^{\ast}_c(K,K_{\mathrm{sing}}) \rightarrow     \mathrm{H}_c^{\ast}(\Sigma,\Sigma_{\mathrm{sing}}\cup \bigcup_{w \in W\setminus \{e\}} wK ),   \]
obtained by cellular excision, is the identity map. This implies that the by inclusion induced map

\[    \mathrm{H}^{\ast}_c(\Sigma,\Sigma_{\mathrm{sing}}) \rightarrow   \mathrm{H}^{\ast}_c(K,K_{\mathrm{sing}})        \]
is surjective. Hence it suffices to show that $\mathrm{H}^{3}(K,K_{\mathrm{sing}}) \otimes_{\mathbb{Z}}\mathbb{Q}\neq 0  $. Using the fact that $K$ is acyclic we conclude that
\[   \mathrm{H}^{3}(K,K_{\mathrm{sing}}) \cong \mathrm{H}^2(K_{\mathrm{sing}}).  \]
Moreover, the Mayer-Vietoris long exact cohomology sequence associated to (\ref{eq: Ksing}) yields an exact sequence

\[   \mathrm{H}^1(L) \rightarrow  H^{1}(L_{\mathrm{sing}}) \rightarrow  \mathrm{H}^{2}(K_{\mathrm{sing}}) .    \]
Since  $\mathrm{H}^1(L)=0$ and $H^{1}(L_{\mathrm{sing}})\cong \mathbb{Z}$, we have $\mathrm{H}^{2}(K_{\mathrm{sing}})\otimes_{\mathbb{Z}}\mathbb{Q}\neq 0$. This concludes the proof of Proposition \ref{groups-Gp}.\qed

\section{The outer space of a free product} \label{sec: outer space}

Suppose that $G_1,\dots,G_r$ are non-trivial, centre-free, finitely generated, and pairwise non-isomorphic one-ended groups, and let
$$G=G_1*\dots*G_r$$
be their free product. Note that the groups $G_1,\dots,G_r$ are infinite and not abelian, and that $G$ is centre-free. In this section we recall a few features of the {\em outer space $\CO$ of free splittings} of $G$ and of the associated action $\Out(G)\actson\CO$. We refer to \cite{GL} (and also \cite{Clay}) for details.

\subsection{Outer space}\label{sec-tired}
As a set, $\CO=\CO(G)$ is the set of all equivariant isometry classes of minimal $G$-actions $G\actson T$ on metric simplicial trees $T$ with trivial edge stabilizers and such that all vertex stabilizers are either trivial or conjugate to one of the groups $G_i$. The triviality of edge stabilizers of the action $G\actson T$ implies that for every $i=1,\dots,r$, there is a unique $G$-orbit of points in $T$ with stabilizer conjugate to $G_i$. We denote the corresponding vertex in $G\backslash T$ by $[G_i]$ and refer to $[G_1],\dots,[G_r]$ as the {\em special} vertices of $G\backslash T$. Note that $G\backslash T$ is a finite tree and that by minimality, all vertices of degree 1 and 2 represent vertices in $T$ with non-trivial stabilizer.

The group $\Aut(G)$ acts on $\CO$ by precomposition, that is 
$$(\phi,\{G\actson T,\ (g,x)\mapsto g\cdot x\})\mapsto\{G\actson T,\ (g,x)\mapsto \phi(g)\cdot x\}.$$
The group $\Inn(G)$ of inner automorphisms acts trivially and hence we obtain a well-defined action $\Out(G)\actson\CO(G)$. If we endow $\CO$ with the equivariant Gromov-Hausdorff topology, then $\Out(G)\actson\CO$ is continuous.

The action $\Out(G)\actson\CO$ is far from being free. In fact, in \cite{GL}, Guirardel and Levitt prove that the stabilizer $\Stab(T)=\Stab_{\Out(G)}(T)$ of $T\in\CO$ is isomorphic to the product (see the remark below)
\begin{equation}\label{gl-stab}
\Stab(T)\simeq\prod_{i=1}^rM_{\deg_{G\bs T}[G_i]}(G_i)
\end{equation}
where $\deg_{G\bs T}[G_i]$ is the degree of the special vertex $[G_i]$ in $G\bs T$ and where 
$$M_{\deg_{G\bs T}[G_i]}=G_i^{\deg_{G\bs T}[G_i]-1}\rtimes\Aut(G_i).$$ 
Note that $M_{\deg_{G\bs T}[G_i]}(G_i)$ also fits into the exact sequence
$$1\to G_i^{\deg_{G\bs T}[G_i]}\to M_{\deg_{G\bs T}[G_i]}(G_i)\to\Out(G_i)\to 1.$$
In particular, it follows that if $\mathrm{Out}(G_i)$ is finite for all $i=1,\ldots,r,$ then $\mathrm{Stab}(T)$ has a finite index subgroup isomorphic to $\prod_{i=1}^rG_i^{\deg_{G\bs T}[G_i]}$.

\begin{bem}
In \cite{GL}, it is only claimed that $\prod_{i=1}^rM_{\deg_{G\bs T}[G_i]}(G_i)$ has finite index in $\Stab(T)$. In the present setting we have equality because we are assuming that the groups are pairwise non-isomorphic.
\end{bem}

Observe that that the stabilizer \eqref{gl-stab} lifts to $\Aut(G)$. We will need this fact for 
$$\deg_{G\bs T}[G_1]=r-1\text{ and }\deg_{G\bs T}[G_2]=\dots=\deg_{G\bs T}[G_r]=1$$ 
and describe briefly the lifting in this concrete situation. For each $i$ let $X_i$ be an Eilenberg-Mac Lane space of type $K(G_i,1)$. Fix $r-1$ distinct points $p_1^2,\dots,p_1^r\in X_1$ and a further point $p_i\in X_i$ for each $i\ge 2$. The space 
$$Z=(X_1\sqcup X_2\sqcup\dots\sqcup X_r)/_{p_i\sim p_1^i\text{ for }i\ge 2}$$
is an Eilenberg-Mac Lane space of type $K(G,1)$ and $M_{r-1}(G_1)\times M_1(G_2)\times\dots\times M_1(G_r)$ is isomorphic to the group of self-homotopy equivalences of $Z$ which map each piece $X_i$ to itself. Since each such homotopy equivalence also fixes the points $p_i=p_1^i$ we get the desired action of $M_{r-1}(G_1)\times M_1(G_2)\times\dots\times M_1(G_r)$ on $G=\pi_1(Z,p_2)$ and thus a homomorphism making the following diagram commute
$$\xymatrix{ & \Aut(G) \ar[d] \\
M_{r-1}(G_1)\times M_1(G_2)\times\dots\times M_1(G_r)\ar[r]\ar[ru] &\Out(G)}$$

Finally, recall that in \cite{GL}, Guirardel and Levitt prove that the outer space $\CO$ is contractible. To do show, they choose some base point $T_0$ and give for all $T\in\CO$ a path $t\mapsto T(t)$ in $\CO$ with $T(0)=T_0$ and $T(1)=T$. It is immediate from their construction \cite[p. 698]{GL} that the whole path is invariant under a subgroup $K\subset\Out(G)$ if both endpoints $T_0$ and $T$ are. In other words, the fixed point set $\CO^K$ is either empty or contractible. In Section \ref{sec: finite ext} we will argue that $\CO^K$ is non-empty and hence contractible when $F$ is a finite subgroup of $\mathrm{Out}(G)$.

\subsection{The spine of outer space}
We now describe the spine $\CS$ of $\CO$ together with its simplicial structure. To begin with let $\CP\CO$ be the set of all $G$-trees $T\in\CO$ such the quotient tree $G\actson T$ has total length $1$. The set $\CP\CO$ is preserved by the action of $\Out(G)$ and in fact there is an $\Out(G)$-equivariant retraction from $\CO$ to $\CP\CO$. In particular, $\CP\CO$ is also contractible.

Given any $G$-tree $T\in\CP\CO$, let $\Delta(T)\subset\CP\CO$ be the set of $G$-trees $T'\in\CP\CO$ which are $G$-equivariantly homeomorphic to $T$, and note that a tree $T'\in\CP\CO$ belongs to $\Delta(T)$ if and only if we can obtain $T'$ from $T$ by equivariantly changing the lengths of the edges. In particular, we can parametrize $\Delta(T)$ by the open simplex $\{\ell\in(0,\infty)^{E(T/G)}\vert\sum_e\ell_e=1\}$ where $E(T/G)$ is the set of edges of the graph $T/G$. Allowing some edges to have length $0$, we have that every vector $\ell\in\{\ell\in\BR_+^{E(T/G)}\vert\sum_e\ell_e=1\}$ in the closed simplex determines a semi-distance $d_\ell$ on $T/G$ and hence a $G$-invariant semi-distance (still denoted $d_\ell$) on $T$. The associated metric space $T_\ell=T/_{d_\ell=0}$ is the $G$-tree obtained by collapsing all degenerate edges to points. The tree $T_\ell$ belongs to $\CP\CO$ if and only if all paths in $T/G$ joining special vertices $[G_i]$ and $[G_j]$ have positive $d_\ell$-length. In other words, the closure $\bar\Delta(T)$ of $\Delta(T)$ in $\CP\CO$ can be identified with the complement in the closed simplex $\{\ell\in\BR_+^{E(T/G)}\vert\sum_e\ell_e=1\}$ of some of its faces. Note that for all $T'\in\sigma_T$ we have that $\bar\Delta(T')\subset\bar\Delta(T)$ is a ``face" of $\bar\Delta(T)$, in the sense that it is the intersection of $\bar\Delta(T)$ with a face of $\{\ell\in\BR_+^{E(T/G)}\vert\sum_e\ell_e=1\}$. 

The sets $\bar\Delta(T)$ behave almost like the simplices of a simplicial structure on $\CP\CO$, but not quite. In order to get an honest such structure one proceeds as follows. Consider first for each $T$ the first barycentric subdivision of the closed simplex $\{\ell\in\BR_+^{E(T/G)}\vert\sum_e\ell_e=1\}$ and let $\sigma_T$ be the subcomplex thereof consisting of simplices which are contained in $\bar\Delta(T)$. The union 
$$\CS=\bigcup_{T\in\CP\CO}\sigma_T\subset\CP\CO$$
has the structure of a simplicial complex, called the {\em spine} of $\CP\CO$. By construction, $\CS$ is an $\Out(G)$-equivariant retract of $\CP\CO$. It follows that the spine $\CS$ is contractible and $\Out(G)$-invariant such that the $\Out(G)$ action on $\CS$ is simplicial and cocompact. Finally, note that for every cell $\sigma$ of $\CS$ contained in $\bar\Delta(T)$,  one has 
\begin{equation}\label{eq-spine-cell-dimension}
\dim\sigma\le\vert E(T/G)\vert-r+1.
\end{equation}

\subsection{Computing dimensions of $\Out(G)$}
Our next goal is to prove the following.

\begin{prop}\label{prop-dimensions}
Let $p_1,\dots,p_r$ be pairwise distinct odd primes, for each $i$ let $G_i=G_{p_i}$ be the group provided by Proposition \ref{groups-Gp}, and set 
$$G=G_1*\dots* G_r.$$
Then $\mathrm{Out}(G)$ is virtually torsion-free with $$\vcd(\Out(G)))\le 5r-5 \ \ \mbox{but} \ \  6r-6 \le \cdm(\Out(G)).$$
\end{prop}

Before launching the proof of Proposition \ref{prop-dimensions} we need to establish the following combinatorial lemma.

\begin{lem}\label{stab-bound}
Suppose that $G_1,\dots,G_r$ are non-trivial, centre-free, finitely generated, and pairwise non-isomorphic one-ended groups, consider their free product $G=G_1*\dots*G_r$, and let $\CO=\CO(G)$ be the associated outer space. We have
$$\sum_i {\deg_{G\bs T}[G_i]} \le 2r-2-\dim(\sigma)$$
for every cell $\sigma$ in the spine $\CS$ of $\CO$ and for every tree $T \in \sigma \setminus \partial \sigma$.
\end{lem}

\begin{proof}
Given any $T\in\sigma\setminus\D\sigma$, let $F\subset G\bs T$ be the forest spanned by the vertices of $G\bs T$ corresponding to vertices of $T$ with trivial $G$-stabilizer. Let $T'\in\CO$ be the tree obtained by collapsing each component of the preimage under $T\to G\bs T$ of $F$ and note that
\begin{equation}\label{eq-being is hungry}
\sum_i {\deg_{G\bs T}[G_i]}=\sum_i {\deg_{G\bs T'}[G_i]}.
\end{equation}
Let $V_0\subset V(G\bs T')$ be the set of those vertices of $G\bs T'$ corresponding to vertices in $T'$ with trivial stabilizer and note that by construction we have $T'\setminus V_0=T\setminus F$. For each $v\in V_0$ choose an edge $e_v\in E(G\bs T')$ adjacent to $v$. If we collapse all the edges in $T'$ which represent one of the edges in $\{e_v\vert v\in V_0\}$, we get a tree $T''$ which is still in $\CO$ but where no further edge can be collapsed. It means that $G\bs T''$ has exactly $r-1$ edges, and hence that $G\bs T'$ and $G\bs T$ have respectively
$$\vert E(G\bs T')\vert=\vert V_0\vert+r-1\text{ and }\vert E(G\bs T)\vert=\vert E(F)\vert+\vert V_0\vert+r-1$$
edges, where $E(F)\subset E(G\bs T)$ is the set of those edges of $G\bs T$ contained in $F$. From \eqref{eq-spine-cell-dimension} we get that the cell $\sigma$ satisfies 
$$\dim\sigma\le \vert E(F)\vert+\vert V_0\vert.$$
Note at this point that, since each vertex of $F$ has degree at least 3, it follows that $T'\setminus V_0$ has at least $3\vert V_0\vert+\vert E(F)\vert$ connected components. Every such component is the union of a subtree of $T'\bs G$ where all vertices are of type $[G_i]$ together with some additional edges joining the subtree to $V_0$. Let $s$ be the total number of additional edges and note that there are at least as many such edges as connected component of $T'\setminus V_0$, so
$$s\ge 3\vert V_0\vert+\vert E(F)\vert.$$
Putting all of this together we have
\begin{align*}
\sum_{i=1}^r\deg_{G\bs T'}([G_i])&=\left(\sum_{v\in V(G\bs T')}\deg_{G\bs T'}(v)\right)-s\\
&=2\vert E(G\bs T')\vert-s\\
&\le 2\vert V_0\vert+2r-2-(3\vert V_0\vert+\vert E(F)\vert)\\
&=2r-2-\vert V_0\vert-\vert E(F)\vert\\
&\le 2r-2-\dim(\sigma).
\end{align*}
The lemma now follows from \eqref{eq-being is hungry}.
\end{proof}

We are now ready to prove Proposition \ref{prop-dimensions}.

\begin{proof}[Proof of Proposition \ref{prop-dimensions}]
First note that it follows from \cite[Th. 6.1]{GL} that $\mathrm{Out}(G)$ is virtually torsion-free and has finite virtual cohomological dimension. Since the spine $\mathcal{S}$ is contractible and $\mathrm{Out}(G)$ acts simplicially and cocompactly on $\mathcal{S}$ with stabilizers that are of type $FP_{\infty}$, it follows that $\mathrm{Out}(G)$ is also of type $FP_{\infty}$. We conclude that 
\[    \mathrm{vcd}(\mathrm{Out}(G))= \max\{ n \in \mathbb{N} \ | \ \mathrm{H}^n(\mathrm{Out}(G),\mathbb{Z}[\mathrm{Out}(G)])\neq 0 \} < \infty.    \]

Denoting $M=\mathbb{Z}[\mathrm{Out}(G)]$, the action of $\mathrm{Out}(G)$ on $\mathcal{S}$ gives rise to is a convergent spectral sequence (e.g. see \cite[Ch.VII.7 ]{brown})
\[  E_1^{p,q}   = \prod_{\sigma \in \Delta^p} \mathrm{H}^q(\mathrm{Stab}(\sigma),M) \rightarrow \mathrm{H}^{p+q}(\mathrm{Out}(G),M). \]
Here $\Delta^p$ is a set of of representatives of $\mathrm{Out}(G)$-orbits of $p$-simplices of $\mathcal{S}$. 

Suppose that we have a simplex $\sigma\in\CS$. By the discussion earlier, the group $\Stab(\sigma)\subset\Out(G)$ has a finite index subgroup isomorphic to $\prod_{i=1}^rG_i^{\deg_{G\bs T}[G_i]}$, where $\deg_{G\bs T}[G_i]$ is the degree of the vertex $[G_i]$ in $G\bs T$. We get from Lemma \ref{wiesollichesnennen} that 
$$\vcd(\Stab(\sigma))=\vcd\left(\prod_{i=1}^rG_i^{\deg_{G\bs T}[G_i]}\right)\le 3\left(\sum_{i=1}^{r}\deg_{G\bs T}[G_i]\right) -r +1.$$
Now, taking into account the bound $\sum_i {\deg_{G\bs T}[G_i]} \le 2r-2-\dim(\sigma)$ provided by Lemma \ref{stab-bound}, we get that
$$\vcd(\Stab(\sigma))\le 5r-5-3\dim(\sigma).$$
It follows in particular that
$$\vcd(\Stab(\sigma))+\dim(\sigma)\le 5r-5.$$

Since  $\mathrm{H}^q(\mathrm{Stab}(\sigma),M)=0$ for all $q>\vcd(\Stab(\sigma))$, we conclude that $E_1^{p,q}=0$ for all $p$ and $q$ with $p+q > 5r-5$. This implies automatically that $E_k^{p,q}=0$ for all $k$ and $p+q>5r-5$. The convergence of the spectral sequence therefore implies that
$$\mathrm{H}^{p+q}(\mathrm{Out}(G),M)=0$$
whenever $p+q>5r-5$, which proves that $\vcd(\Out(G)))\le 5r-5$.

Note now that there is a $G$-tree $T\in\CS$ such that $[G_1],\dots,[G_r]$ are the only vertices of $G\bs T$, where $[G_2],\dots,[G_r]$ are leaves and where $[G_1]$ is connected to $[G_i]$ for all $i\ge 2$. It follows that $n_1=r-1$ and $n_2=\dots=n_r=1$. In particular, $\Stab(T)$ contains a copy of the direct product $G_1^{r-1}\times G_2\times\dots\times G_r$. The second claim of Lemma \ref{wiesollichesnennen} implies that
$$\cdm(\Out(G))\ge\cdm(\Stab(T))\ge\cdm(G_1^{r-1}\times G_2\times\dots\times G_r)=6r-6$$
which is what we needed to prove.
\end{proof}

\section{Dimension rigidity and free products}

The aim of this section is to prove the following proposition and to finalize the proof of the main theorem.

\begin{prop} \label{thm-free-product-rigid}
Let $G_1,\ldots,G_r$ be a collection of one-ended, finitely presented, virtually torsion-free, pairwise non-isomorphic groups. If the groups $G_i$ are dimension rigid, then their free product $G=G_1\ast \ldots \ast G_r$ is dimension rigid as well. 
\end{prop}

The proof of this proposition will require an analyses of the fixed point sets $\CO^F$ for finite subgroups $F$ of $\mathrm{Out}(G)$.

\subsection{Fixed points in $\CO$ and extensions} \label{sec: finite ext}
Suppose from now on that the groups $G_1,\dots,G_r$ are one-ended, finitely presented and pairwise non-isomorphic groups. Let $G=G_1*\dots*G_r$ be their free product and consider an extension
$$1\to G\to\Gamma\to F\to 1$$
with $F$ finite. Since $\Gamma$ is quasi-isometric to $G$, it follows that $\Gamma$ is finitely presented and has infinitely many ends. It therefore follows from Dunwoody's accessibility theorem that $\Gamma$ acts cocompactly on a simplicial tree $T$ such that all edge stabilizers are finite and all vertex stabilizers have at most one end. 
Since all the $G_i$ are one-ended it follows that each $G_i$ is contained in a vertex stabilizer of $T$. Moreover, the assumption that $G$ has finite index in $\Gamma$, together with the fact that $\Gamma$-vertex stabilizers in $T$ are one-ended implies that $G_i$ and $gG_jg^{-1}$ have a common fixed-point in $T$ if and only if $i=j$ and $g\in G_i$. We have proved that the induced graph of groups decomposition of $G$ has the groups $G_i$ as vertex groups. In particular, $G\actson T\in\CO$ and $\Gamma\bs T$ is a finite tree. We deduce that $\Gamma$ can be written as an amalgam
$$\Gamma \cong \left(\dots\left(\Gamma_1 \ast_{F_1}\Gamma_2\right)\ast_{F_2} \ldots\right) \ast_{F_{r-1}} \Gamma_r $$ 
such that the $F_1,\ldots,F_{r-1}$ are finite and for each $i=1,\ldots,r$, the group  $\Gamma_i$ contains (up to possibly relabeling) $G_i$ as finite index normal subgroup. We summarize our observations in the following lemma.

\begin{lem} \label{extension-iterated-amalgamated}  
If $\Gamma$ is an extension of a finite group $F$ by $G_1\ast \ldots \ast G_r$, then 
$$\Gamma \cong \left(\dots\left(\Gamma_1 \ast_{F_1}\Gamma_2\right)\ast_{F_2} \ldots\right) \ast_{F_{r-1}} \Gamma_r $$ 
such that the $F_1,\ldots,F_{r-1}$ are finite and (after possibly relabeling) $\Gamma_i$ contains $G_i$ as finite index normal subgroup for each $i=1,\ldots,r$.\qed
\end{lem}

Suppose that $F\subset\Out(G)$ is finite and consider its preimage $\tilde F$ under the homomorphism $\Aut(G)\to\Out(G)$. Applying the argument above to $\Gamma=\tilde F$ we get a tree $T\in\CO$ on which not only $G$ but also the whole group $\tilde F$ acts isometrically. In other words, $T$ is a fix point of the finite group $F$, meaning that $\CO^F\neq\emptyset$ and hence contractible. Since the spine $\mathcal{S}$ is an $\mathrm{Out}(G)$-equivariant retract of $\CO$, we can immediately conclude the following.

\begin{lem}\label{lem: contr fixed} 
The fixed point set $\mathcal{S}^F$ is contractible for every finite subgroup $F$ of $\mathrm{Out}(G)$. \qed
\end{lem}

The argument to prove Lemma \ref{lem: contr fixed}, and thus basically also the proof of Lemma \ref{extension-iterated-amalgamated}, was explained to us by Vincent Guirardel.

\subsection{Proof of Proposition \ref{thm-free-product-rigid}}

Starting with the proof of the proposition, suppose that we have an extension $\Gamma$ of a finite group by $G$. Our goal is to construct a model for $\underline E\Gamma$ of dimension 
$$\vcd(G)=\max \{ \vcd(G_i) \  | \ i=1,\ldots,r \}.$$
By Lemma \ref{extension-iterated-amalgamated} we know that $\Gamma$ is of the form
$$\Gamma \cong \left(\dots\left(\Gamma_1 \ast_{F_1}\Gamma_2\right)\ast_{F_2} \ldots\right) \ast_{F_{r-1}} \Gamma_r $$ 
such that  each is $F_i$ is finite  and each $\Gamma_i$ contains $G_i$ as finite index normal subgroup.  Now define a chain of subgroups $\Gamma^{(1)}\subset\Gamma^{(2)}\subset\dots\subset\Gamma^{(r)}=\Gamma$ as follows
$$\Gamma^{(1)}=\Gamma_1\text{ and }\Gamma^{(i)}=\Gamma^{(i-1)}*_{F_{i-1}}\Gamma_i.$$ 
We are going to prove by induction that 
$$\gdim(\Gamma^{(i)})=\max\{\vcd(G_1),\dots,\vcd(G_i)\}$$ 
for $i=1,\dots,r$. The base case follows because $\Gamma^{(1)}=\Gamma_1$ is an extension of a finite group by $G_1$, and the latter is assumed to be dimension rigid. It thus suffices to establish the induction step. Since $\Gamma^{(i)}$ is the amalgamated product over the finite group $F_{i-1}$ of $\Gamma^{(i-1)}$ and $\Gamma_i$, and since 
\begin{align*}
\gdim(&\Gamma^{(i-1)})=\max\{\vcd(G_1),\dots,\vcd(G_{i-1})\}\text{ by induction, and}\\
\gdim(&\Gamma_i)=\vcd(G_i)\text{ because }G_i\text{ is dimension rigid,}
\end{align*}
we see that the induction step and hence also the proposition follow from Lemma \ref{free-gdim} below. \qed

\begin{lem}\label{free-gdim}
If $G=G_1*_FG_2$ is a group which arises as the proper amalgamated product over a finite group $F$ of two groups $G_1,G_2$, then we have 
$$\gdim(G)=\max\{\gdim(G_1),\gdim(G_2),1\}.$$
Here proper means that $F\neq G_1$ and $F\neq G_2$.
\end{lem}

Before launching the proof of Lemma \ref{free-gdim} we make two observations. First, note that in the case that one of the groups $G_1,G_2$ is infinite, the claim reduces to 
$$\gdim(G)=\max\{\gdim(G_1),\gdim(G_2)\}.$$
The case that both factors $G_i$ are finite is well-known for in this case $G$ is virtually free and thus has geometric dimension $1$. Secondly, we point out that the corresponding cohomological statement $\underline{\mathrm{cd}}(G)=\max\{\underline{\mathrm{cd}}(G_1),\mathrm{cd}(G_2),1\}$ follows from \cite[Cor. 8]{DemPetTal}. The lemma follows from this cohomological statement, as long as $\underline{\mathrm{cd}}(G)\geq 3$, which covers the case in which we will need Lemma \ref{free-gdim}. Below the interested reader can find a geometric proof of  Lemma \ref{free-gdim}, in general.

\begin{proof}
For $i=1,2$ let $Z_i$ be a model of $\underline EG_i$ of dimension
$$\dim(Z_i)=\gdim(G_i).$$
Choose a point $p_i\in Z_i$ fixed by the finite subgroup $F\subset G_i$. Finally, let $X_i$ be a model for $EG_i$ and note that $X_i\times Z_i$, endowed with the diagonal $G_i$ action, is also a model for $EG_i$. Let $X_F$ for $EF$ and note that, up to modifying $X_i$, we may assume without loss of generality that there is an honest equivariant embedding $\iota_i:X_F\to X_i$.

We get a model $X$ for $EG$ as the universal cover of the $K(G,1)$ obtained as follows
$$((X_1\times Z_1)/G_1)\bigsqcup((X_F/F)\times[0,1])\bigsqcup((X_2\times Z_2)/G_2))\big/_\sim$$
where $\sim$ is the equivalence relation generated by setting 
$$([x],0)\sim [(\iota_1(x),p_1)]\text{ and }([x],1)\sim [(\iota_2(x),p_2)]\text{ for all }x\in X_F.$$

By construction, the space $X$ is the union of copies of $X_1\times Z_1$, $X_2\times Z_2$ and $X_F\times[0,1]$. Denote by $\sim_{BS}$ the equivalence relation on $X$ generated by asserting that (1) any two points in a copy of the $X_i\times Z_i$ are equivalent, and (2) any two points of the form $(x,t)$ and $(x',t)$ in any copy of $X_F\times[0,1]$ are equivalent. Then $T=X/_{\sim_{BS}}$ is nothing but the Bass-Serre tree of the splitting $G=G_1*_FG_2$. 

Note also that $X$ is covered (one could think foliated) by disjoint copies of the spaces $X_1,X_F,X_2$. We will refer to each one of these spaces as a leaf, and let $\sim_\CF$ be the equivalence relation generated by declaring any two points in the same leaf to be equivalent. We claim that the space $Z=X/_{\sim_\CF}$ is the desired model for $\underline EG$. To begin with note that the set of leaves of $X$ is invariant under the action $G\actson X$. We thus get an induced action $G\actson Z$. In fact, since $\sim_{BS}$ is coarser than $\sim_\CF$ we have a commutative diagram of $G$-equivariant maps
$$\xymatrix{ X \ar[d] \ar[dr] &  \\ Z=X/_{\sim_\CF} \ar[r]^p & T=X/_{\sim_{BS}}.}$$
Suppose now that $K\subset G$ is a finite group and note that $K$ fixes a vertex $v$ in $T$. It follows that $K$ is contained in a conjugate of $G_1,G_2$. Say for the sake of concreteness that $K\subset G_1$, meaning that $v$ is the vertex of $T$ corresponding to the subgroup $G_1$. Then $K$ acts on $p^{-1}(v)=Z_1$ and has a fixed point there in because $Z_1$ is a model for $\underline EG_1$. Having proved that $Z^K\neq\emptyset$ for all $K\subset G$ finite, it remains to be proved that $Z^K$ is contractible. Noting that $(p^{-1}(v))^K$ is contractible for every vertext $v\in T$ fixed by $K$, we get that $Z^K$ is a tree of contractible spaces and thus contractible. 

Having proved that $Z$ is a model for $\underline EG$ note that 
$$\dim(Z)=\max\{\dim(Z_1),\dim(Z_2),1\}=\max\{\gdim(G_1),\gdim(G_2),1\}.$$
This proves that 
$$\gdim(G)\le\max\{\gdim(G_1),\gdim(G_2),1\}.$$
However, the other inequality also holds because $G_1,G_2\subset G$ and because the amalgamated product is non-trivial.
\end{proof}

\subsection{Proof of Theorem \ref{main2}} We conclude this paper by finishing the proof of the main theorem. Let $p_1,\dots,p_{r+1}$ be pairwise distinct odd primes and let $G_i=G_{p_i}$ be the groups provided by Proposition \ref{groups-Gp}. Note that $G_i$ and $G_j$ are isomorphic if and only if $i=j$. We will consider the free product
$$G=G_1*\dots*G_{r+1}.$$
Each one of the groups $G_i$ is CAT(0) and Gromov-hyperbolic and has finite outer automorphism group. In particular $G$ is also CAT(0) and Gromov-hyperbolic. Moreover, each one of the groups $G_i$ is dimension rigid, finitely presented (since Gromov-hyperbolic), and one-ended. In particular, we get from Proposition \ref{thm-free-product-rigid} that $G$ is also dimension rigid.

Noting that we have $r+1$ factors, we get from Proposition \ref{prop-dimensions} that $\Out(G)$ is virtually torsion-free with $\vcd(\Out(G)))\le 5r$ but $\cdm(\Out(G))\ge 6r$, so we conclude that
$$\vcd(\Out(G))\le \gdim(\Out(G))-r$$
as desired.

It remains to show that $\Out(G)$ admits a cocompact model for $\underline{E}\Out(G)$. To this end, first note that the proof of \cite[Prop. 5.1]{LuckWeiermann} in fact proves the following result: Let $\Gamma$ be a group and $Z$ a cocompact $\Gamma$-CW-complex such that $Z^F$ is contractible for every finite subgroup $F$ of $\Gamma$ and such that the stabilizers of the action of $\Gamma$ on $Z$ all admit cocompact models for proper actions. Now let $X$ be a model for $\underline{E}\Gamma$ (which always exists, e.g. by \cite[Th. 1.9.]{Luck2}), then  $Z \times X$, equipped with the diagonal $\Gamma$-action is $\Gamma$-equivariantly homotopy equivalent to a cocompact $\Gamma$-CW-complex. Since $Z \times X$ is also a model for $\underline{E}\Gamma$, we conclude that $\Gamma$ admits a cocompact model for $\underline{E}\Gamma$.

Since the action of $\Out(G)$ on the spine $\mathcal{S}$ of the outer space $\CO$ associated to $G_1 \ast \ldots \ast G_{r+1}$ is cocompact and the fixed point set $\mathcal{S}^F$ is contractible for every finite subgroup $F$ of $\mathrm{Out}(G)$ by Lemma \ref{lem: contr fixed}, we conclude that in order to show that $\Out(G)$ admits a cocompact model for $\underline{E}\Out(G)$, it suffices to show that that the stabilizers of the action of $\Out(G) $ on $\mathcal{S}$ admit a cocompact model for proper actions.

 As we discussed earlier, the stabilizers of this action are of the form $\prod_{i=1}^{r+1} M_{n_i}(G_i)$ where $M_{n_i}(G_i)$ fits into a short exact sequence

\[    1 \rightarrow G_i^{n_i}\rightarrow M_{n_i}(G_i) \rightarrow \Out(G_i) \rightarrow 1.   \]
Since $\Out(G_i)$ is finite and $G_i$ is Gromov-hyperbolic, we conclude that $\prod_{i=1}^{r+1} M_{n_i}(G_i)$ is finite index subgroup of a finite product of Gromov-hyperbolic groups.
Since any Gromov-hyperbolic group admits a cocompact model for proper actions by \cite{MeintrupSchick}, we are done for $\Out(G)$.

We comment now briefly on the proof for $\Aut(G)$. First note that $\Aut(G)$ is virtually torsion free. For instance, this follows from the fact that $\Aut(G)$ is a subgroup of the virtually torsion free subgroup $\Out(G*\BZ/2\BZ)$. Since $\Aut(G)$ is virtually torsion free we get from the exact sequence $$1\to G\to\Aut(G)\to\Out(G)\to 1$$ that 
$$\vcd(\Aut(G))\le\vcd(G)+\vcd(\Out(G))\le 3+5r.$$
To get a lower bound for the geometric dimension we proceed as in the proof of Proposition \ref{prop-dimensions}. Let $T\in\CS$ be a $G$-tree such that $[G_1],\dots,[G_{r+1}]$ are the only vertices of $G\bs T$, where $[G_2],\dots,[G_{r+1}]$ are leaves and where $[G_1]$ is connected to $[G_i]$ for all $i\ge 2$. As we mentioned in Section \ref{sec-tired}, the stabilizer 
$$\Stab(T)=M_{r}(G_1)\times M_1(G_2)\times\dots\times M_1(G_{r+1})$$
of $T$ lifts to $\Aut(G)$. Moreover, $\Stab(T)$ contains the group $G_1^{r}\times G_2\times\dots\times G_{r+1}$. As in Proposition \ref{prop-dimensions} we get thus that $\cdm(\Stab(T))\ge 6r$ and thus that
$$\gdim(\Aut(G))=\cdm(\Aut(G))\ge 6r.$$
The claim follows thus when we use $r+4$ factors instead of only $r+1$.
\qed

\end{document}